\documentclass[11pt]{amsart}
\usepackage{amssymb, amsmath, amsthm}
\usepackage[latin1]{inputenc}
\usepackage{graphicx}
\usepackage[final]{hyperref}
\usepackage[a4paper, centering]{geometry}
\usepackage{color}
\usepackage{graphicx} 

\newtheorem{theorem}{Theorem}
\newtheorem{proposition}[theorem]{Proposition}
\newtheorem{lemma}[theorem]{Lemma}
\newtheorem{corollary}[theorem]{Corollary}
\theoremstyle{definition}

\theoremstyle{remark}
\newtheorem{remark}[theorem]{Remark}
\newtheorem{example}[theorem]{Example}
\def\R{\mathbb{R}}

\def\pscal#1#2{\left\langle#1,\,#2\right\rangle}

\def\dist{d}

\font\maius=cmcsc10 scaled1200

\DeclareMathOperator{\Cut}{\overline \Sigma}  
\DeclareMathOperator{\high}{M}
\DeclareMathOperator{\conv}{conv}

\begin{document}

\title[A symmetry problem for the infinity Laplacian]%
{A symmetry problem for the infinity Laplacian}%
\author[G.~Crasta, I.~Fragal\`a]{Graziano Crasta,  Ilaria Fragal\`a}
\address[Graziano Crasta]{Dipartimento di Matematica ``G.\ Castelnuovo'', Univ.\ di Roma I\\
P.le A.\ Moro 2 -- 00185 Roma (Italy)}
\email{crasta@mat.uniroma1.it}

\address[Ilaria Fragal\`a]{
Dipartimento di Matematica, Politecnico\\
Piazza Leonardo da Vinci, 32 --20133 Milano (Italy)
}
\email{ilaria.fragala@polimi.it}

\date{\today}

\begin{abstract}  
Aim of this paper is to prove necessary and sufficient conditions on the geometry of a domain $\Omega \subset \R ^n$ in order that 
the homogeneous Dirichlet problem for the infinity-Laplace equation in $\Omega$ with constant source term admits  a viscosity solution depending only on the distance from $\partial \Omega$. 
This problem was previously addressed and studied by  Buttazzo and Kawohl in \cite{butkaw}. 
In the light of some geometrical achievements reached in our recent paper \cite{CFb},
we revisit the results obtained in \cite{butkaw}  and we prove strengthened versions of them, 
where any regularity assumption on the domain and on the solution is removed.  
Our results require a delicate analysis based on viscosity methods. 
In particular, we need to build  suitable viscosity test functions, whose construction involves 
a new estimate of the distance function $d_{\partial \Omega}$ near singular points.
\end{abstract} 


\maketitle

\section{Introduction}\label{secintro}

In \cite{butkaw}, Buttazzo and Kawohl considered the following overdetermined boundary value problem
for the infinity Laplacian:
\begin{equation}\label{overdet}
\left\{\begin{array}{ll}								
-\Delta_\infty u=1\quad &\mbox{in }\Omega\\
u=0\quad &\mbox{on }\partial\Omega\\
{\frac{\partial u}{\partial \nu}}=c\quad & \mbox{on }\partial\Omega\, .
\end{array}\right. 
\end{equation}
Here  $\Omega$ is an open bounded domain of $\R^n$ with a smooth boundary, $\nu$ denotes the unit inner normal to $\partial \Omega$, and $c$ is a positive constant. 
We recall that the
infinity Laplacian operator $\Delta_\infty$ is defined on smooth functions by
$$\Delta_\infty u=\langle D^2 u \nabla u,\nabla u\rangle=
\sum_{i,j=1}^n\frac{\partial^2u}{\partial x_ix_j}\frac{\partial
u}{\partial x_i}\frac{\partial u}{\partial x_j} \qquad\mbox{for all
}u\in C^2(\Omega).$$
In the last decade, pde's involving this operator, first discovered by Aronsson in the pioneering work \cite{Aro},  have attracted an increasing amount of interest; without any attempt of completeness, we refer to the monograph \cite{Barron} and to the representative works \cite{ArCrJu, BhMo, Cran, CEG, EvYu, Jen, JenWanYu, LuWang, Yu}.  

In view of the identity
$$\Delta_pu=\nabla\cdot(|Du|^{p-2}Du)=(p-2)|Du|^{p-4}\left(\Delta_\infty u+\frac{|Du|^2\, \Delta u}{p-2}\right)\, $$
and of standard convergence results for viscosity solutions (see for instance \cite{CHL}), if a sequence of $p$-harmonic functions converges locally uniformly as $p \to + \infty$, the limit function $u$ is an 
infinity-harmonic function, i.e.\ a solution to $\Delta _ \infty u = 0$. 
This is the reason why, with a mathematical abuse, (\ref{overdet}) can be regarded as the limit as $p \to + \infty$ of
the overdetermined boundary value problems
\begin{equation}\label{pLap}
\left\{\begin{array}{ll}
- \Delta _p u=1\quad &\mbox{in }\Omega\\
u=0\quad &\mbox{on }\partial\Omega\\
{\frac{\partial u}{\partial \nu}}=c\quad & \mbox{on }\partial\Omega\, .
\end{array}\right.
\end{equation}
The latter have been extensively studied in the literature. 
The first result was proved by Serrin in the seminal paper \cite{Se} and states that, in the linear case of the Laplacian (namely when $p=2$), problem (\ref{pLap}) admits a solution if and
only if $\Omega$ is a ball.
Since then, several  generalizations and related results have been proved, see for instance  
\cite{BN,DP,GL,K2,fg,fgk}. 

The methods adopted in the literature to treat problem (\ref{pLap}) are no longer exploitable when dealing with problem (\ref{overdet}), 
because the infinity Laplacian operator $\Delta _\infty$ is highly degenerate.  
In particular,  solutions can be no longer intended either in classical or in weak sense, respectively because they are not expected to be of class $C^2$ (cf.\  \cite{EvSav,EvSm}), and because $\Delta _\infty$ is not in divergence form. 
Thus, the notion of solution has to be understood in the sense of viscosity (the definition is recalled for convenience at the end of the Introduction).  Moreover, as long as one wants to understand  both the boundary conditions in (\ref{overdet}) pointwise, one has to restrict attention to solutions which are $C ^1$ up to the boundary. 

\par

These difficulties led to consider a simplified version of problem (\ref{overdet}), which consists in investigating the existence of  viscosity solutions to the Dirichlet problem 
\begin{equation}
\label{f:dirich}
\begin{cases}
-\Delta_{\infty} u = 1 &\text{in}\ \Omega\\
u = 0 &\text{on}\ \partial\Omega 
\end{cases}
\end{equation}
within the  class of functions depending only
on the distance to the boundary of $\Omega$, namely functions having the same level sets as the distance function
\begin{equation}\label{defd}
d_{\partial \Omega}(x) := \min_{y\in \partial \Omega} |x-y|,\qquad x\in\overline{\Omega}\ .
\end{equation}

The existence and uniqueness of a viscosity solution to problem (\ref{f:dirich}) (actually of a more general version of it, allowing a constant-sign source term) has been established by Lu and Wang in \cite{LuWang}. 
The problem is then to establish for which geometries of $\Omega$ such solution turns out to depend only on $d_{\partial \Omega}$, and in particular whether this occurs only if $\Omega$ is a ball. 
Following \cite{gazzola}, by {\it web functions} in the sequel  we mean continuous functions depending only on $d_{\partial \Omega}$   (the name comes from the fact that, in case of planar polygons,
level lines of the distance functions recall the pattern of a spider web). To the best of our knowledge, these functions 
firstly appeared in  the monograph by P\'olya and Szeg\"o
\cite[Section 1.29]{posz}; more recently, they have found application in different variational problems, see \cite{Cf,Cg,CFGa,CFGc,CFGb,CGa}. 

Clearly, asking that the solution of problem (\ref{f:dirich}) is a web function
is a more severe restriction than imposing just the constancy of its normal derivative along the boundary as in (\ref{overdet}). 
However, this restriction is somehow natural, for instance because it is known that $d_{\partial \Omega}$ is the uniform limit as $p \to + \infty$ of
 the solution $u_p$ to the first two eqs.\ in problem (\ref{pLap}) (see \cite{BDM,K1}), as well as the unique infinity ground state on $\Omega$ up to constant factors (see \cite{Yu}). This latter result holds under the restriction that
 the cut locus and high ridge of $\Omega$ coincide. Such geometric property is precisely
   the same found by Buttazzo and Kawohl  as a necessary and sufficient condition for the existence of a web solution to problem (\ref{f:dirich}). 
   
Let us recall that the {\it cut locus} and the {\it high ridge} of $\Omega$ are defined respectively as 
\begin{eqnarray} 
\hbox{$\Cut(\Omega)$ := the closure of the singular set  $\Sigma (\Omega)$
of $d_{\partial \Omega}$} & \label{cut} \\ \noalign{\smallskip}
\hbox{$\high (\Omega)$ := the set where $d _{\partial \Omega}(x) = \rho _\Omega:= \max _{ \overline \Omega} d _{\partial \Omega}\,. $ } & \label{high}
\end{eqnarray}

Moreover, let us introduce the function $\phi _\Omega$ which is the natural candidate  to be a web solution to (\ref{f:dirich}), as it can be easily seen via a one-dimensional ansatz (cf. \cite[Section 2]{butkaw}):
\begin{equation}\label{defphi} 
\phi _\Omega (x) : = c_0 \left[\rho_\Omega ^{4/3} - (\rho_\Omega - \dist_{\partial \Omega}(x))^{4/3}\right] \,, \qquad \hbox{ where } c_0 := 3^{4/3} / 4\,.
\,  
\end{equation}

With this notation, the result by Buttazzo and Kawohl reads:

\begin{theorem}\label{teoweb} \cite[Theorem 1]{butkaw}

Let $\Omega\subset \R^n$ be an open bounded connected domain,
with $\partial \Omega$ of class $C^2$.
\begin{itemize}
\item[(a)] Assume that $\Cut (\Omega) =\high(\Omega) $.
 Then $\phi
_\Omega$ is the unique web viscosity solution of class $C ^ 1
(\overline \Omega)$ to problem \eqref{f:dirich}.

\smallskip
\item[(b)] Conversely, assume that problem $(\ref{f:dirich})$ admits
a web viscosity solution of class $C ^ 1 (\overline \Omega)$. Then
$\Cut(\Omega) = \high(\Omega)$.
\end{itemize}
\end{theorem}

To some extent surprisingly, this result seems to indicate in particular that symmetry does not hold for the problem under study, namely that there exists some regular domain, different from a disk, where  (\ref{f:dirich}) admits a web viscosity solution. 
Actually the examples given in \cite{butkaw} of non-spherical domains $\Omega$ with
$\Cut(\Omega ) = \high (\Omega)$ are  of the form
$$
\Omega _\gamma := \big \{ x \in \R ^ 2 \ :\ d_\gamma (x) < r \big \}\, ,
$$
where $d_\gamma$ is the distance from a $C ^ {1,1}$-curve $\gamma:[ 0, L]  \to \R ^2$,   with $\gamma (0) \neq \gamma (L)$.   

The starting point of our investigation is the observation that, in fact, none of the domains  $\Omega _\gamma$ can have a $C ^2$ boundary, unless $\gamma$ is a singleton and $\Omega_\gamma$ is a disk.  
More generally, the simultaneous validity of the two conditions  
$\Cut(\Omega ) = \high (\Omega)$ and $\partial \Omega \in C ^2$, implies that $\Omega$ is a ball as soon as 
$\Omega \subset \R ^2$ is simply connected, or $\Omega\subset \R ^n$ is convex. 
This follows from some geometrical results we proved in a recent paper (see \cite[Thm.\ 6 and Thm.\ 12]{CFb}).  

In this perspective, it is natural to inquire about the validity of Theorem \ref{teoweb} when no regularity assumptions on the domain $\Omega$ and on the solution $u$ are made. Our results provide a complete answer to this question and can be summarized as follows:

\begin{itemize}
\item[$\bullet$]  Assume that  $\Omega$ is an open bounded domain, satisfying $\Cut (\Omega) = \high (\Omega)$ (and no  further  regularity requirement). Then the function $\phi _\Omega$ is still the unique solution to problem (\ref{f:dirich})  (see Theorem \ref{t:a}). 

\medskip
\item[$\bullet$] Assume that problem (\ref{f:dirich})  admits a web viscosity solution $u$ (which a priori has no further  regularity besides continuity). 
Then $u = \phi _\Omega$ and there holds
$\Cut (\Omega) = \high (\Omega)$ (see Theorem \ref{t:b}). Consequently, if the space  dimension is $n=2$, $\Omega$ has the special form 
of a tubular neighborhood around a $C ^ {1,1}$ manifold, so that it fails in general to have radial symmetry (see Corollary \ref{corgeo1}).  In spite, if one assumes that $\partial \Omega \in C ^2$, then $\Omega$ must be necessarily a ball under the following additional restrictions:  either $\Omega$ is convex, or $n = 2$ and $\Omega$ is simply connected  (see Corollary \ref{corgeo2}). \end{itemize}

\medskip
We advertise that these results cannot be obtained by minor modifications of the arguments used in \cite{butkaw} to prove  Theorem \ref{teoweb}, but require a delicate analysis based on viscosity methods. 
In particular, the proof of Theorem \ref{t:b} relies on the construction of suitable viscosity test functions and involves  a result which may have an autonomous interest, that is a new estimate of the distance function $d _{\partial \Omega}$ near singular points (see Theorem \ref{p:estid}). 

It remains by now an open problem, which seems to be quite challenging,  to establish whether the conclusion $\Cut (\Omega) = \high (\Omega)$ of Theorem \ref{t:b} remains valid under the weaker assumption that the overdetermined problem (\ref{overdet}) admits a solution.
A major difficulty to deal with this problem is the lackness of any information about the regularity properties of the solution $u$ to \eqref{f:dirich} beyond the local Lipschitz regularity. In particular, the expected regularity is not higher than that expected for infinity-harmonic functions, namely  $C ^ {1, \alpha}$ regularity. (At present, such regularity has been settled only in two space dimensions by Evans and Savin in \cite{EvSav}; let us also recall that infinity-harmonic functions turn out to be  everywhere differentiable in arbitrary space dimensions, see \cite{EvSm}.)  The study of such regularity issues for the solution to problem (\ref{f:dirich}), as well as the investigation of its possible concavity-like properties, are in our opinion interesting topics for further research.

 \bigskip
{\maius Outline of the paper.}  
In Section \ref{seca} we deal with the sufficiency of the condition $\Cut (\Omega) = \high (\Omega)$ for the existence of web solution to problem (\ref{f:dirich}). Necessity is more delicate and it is proved in Section \ref{secb}, relying on a geometric result given in Section \ref{secgeo}. 
Finally in the Appendix we show how the proof of necessity 
can be considerably simplified if the solution is assumed a priori to be differentiable (which might happen to be not restrictive in the light of the results in  \cite{EvSm,EvSav}).

\bigskip
{\maius Notation.}
Throughout the paper,  
$\Omega$ will always denote a non-empty open bounded domain
of $\R^n$.
A point $x\in\Omega$ will be called \textsl{regular}
if the distance function from the boundary
$d_{\partial\Omega}$ is differentiable at $x$,
and \textsl{singular} otherwise.
The \textsl{singular set} of $\Omega$ (or of $d_{\partial\Omega}$), 
i.e.\ the set of all singular points of $\Omega$, will be
denoted by $\Sigma(\Omega)$.
We shall denote by $\Cut (\Omega)$ and $\high (\Omega)$ the sets introduced respectively in  (\ref{cut}) and (\ref{high}), by $\rho _\Omega$ the maximum of $d _{\partial \Omega}$ on $\overline \Omega$, and by  $\phi _\Omega$ the function defined in (\ref{defphi}). 
Moreover, we set
\begin{equation}\label{defg}
g(t) := c_0 \left[ \rho_\Omega^{4/3} - (\rho_\Omega - t)^{4/3}
\right], \qquad
t \in [0, \rho_\Omega],
\end{equation}
so that the function $\phi _\Omega$ can also be rewritten as
$$\phi_{\Omega}(x) = g(d_{\partial \Omega}(x))\, ,  \qquad x\in\overline{\Omega}\,. $$

Following \cite{CHL},  
a viscosity solution  to the equation $-\Delta_\infty u -1 = 0$
is a function
$u\in C^0({\Omega})$  which is both a viscosity
subsolution, i.e.
\begin{equation}\label{f:subsol}
-\Delta_\infty\varphi(x_0)-1\leq 0\
\quad
\text{whenever}\
\varphi\in C^2(\Omega)\
\text{and $\varphi-u$ has a local minimum at $x_0$},
\end{equation}
and a viscosity super-solution, i.e. 
\begin{equation}\label{f:supersol}
-\Delta_\infty\varphi(x_0)-1\geq 0\
\quad
\text{whenever}\
\varphi\in C^2(\Omega)\
\text{and $\varphi-u$ has a local maximum at $x_0$}.
\end{equation}
By a viscosity solution to the Dirichlet problem \eqref{f:dirich}
we mean a function $u\in C^0(\overline{\Omega})$ such that
$u=0$ on $\partial\Omega$ and $u$ is a viscosity solution
to $-\Delta_\infty u = 1$ in $\Omega$.

\bigskip
{\maius Acknowledgments.}
The authors would like to thank Italo Capuzzo Dolcetta 
for some useful discussions.

\section{Sufficiency of the condition $\Cut (\Omega) = \high (\Omega)$. } \label{seca}

In this section we prove: 

\begin{theorem}
\label{t:a}
Let $\Omega\subset\R^n$ be an open bounded connected domain,
satisfying $\Cut(\Omega) = \high(\Omega)$.
Then the function $\phi_{\Omega}$ is the unique viscosity solution to the Dirichlet boundary value problem $(\ref{f:dirich})$. 

Moreover, if $\partial \Omega$ is of class $C^1$,
then the function $\phi_{\Omega}$ is also
the unique viscosity solution
to the overdetermined boundary value problem \eqref{overdet}, with $c = (3 \rho_{\Omega})^{1/3}$.
\end{theorem}

\begin{remark}
Notice that the validity of the condition $\partial \Omega$ of class $C ^1$ is not guaranteed by the coincidence of cut locus and high ridge. A simple example of domain $\Omega$ with $\Cut (\Omega) = \high (\Omega)$ but $\partial \Omega \not \in C ^1$ can be found in \cite[Remark 7]{CFb}.  
However, under the sole assumption $\Cut (\Omega) = \high (\Omega)$, the existence of a solution to the overdetermined boundary problem (\ref{overdet}) can be inferred up to replacing $\Omega$ by a parallel set. More precisely, set $S:= \Cut (\Omega) = \high (\Omega)$ and, for $r>0$,  denote by $S_r$ the set of points with distance from $S$ less than $r$.  It turns out that, for $r$ sufficiently small, $\Cut (S_r) = \high (S_r)$ and $\partial S _ r \in C^1$ (see \cite[Proposition 13 and Lemma 16]{CFb}). Hence, for such values of $r$, there exists a unique viscosity solution
to the overdetermined boundary value problem \eqref{overdet} on $S_r$, given by the corresponding function $\phi _{S_r}$.  
\end{remark}

For the proof of Theorem \ref{t:a} we need the following simple lemma.

\begin{lemma} \label{l:primo}
Let $\Omega\subset\R^n$ be an open bounded connected domain
and let $g$ be defined by $(\ref{defg})$. For every fixed point $x_0 \in \high(\Omega)$, the radial-profile function
\begin{equation}\label{defv}
v(x) := g(\rho_\Omega  - |x-x_0|)\, , \qquad x \in {\Omega} \,. 
\end{equation}
is a viscosity solution the equation
to $-\Delta_{\infty} v = 1$ in $\Omega$. 
\end{lemma}

\proof We observe that $v$ is a classical solution to the equation $-\Delta_{\infty} v = 1$ on the set $\Omega\setminus\{x_0\}$, because it is therein of class $C^2$, with 
\[
-\Delta_{\infty}v (x) = -g''(\rho_\Omega -|x-x_0|)\,
\left[g'(\rho_\Omega -|x-x_0|)\right]^2 = 1 \qquad
\forall x \in \Omega \setminus \{x_0\}. 
\]
On the other hand, by the assumption 
$d (x_0) = \rho_\Omega$, the conditions $g' ( \rho_\Omega ) = 0$ and $g'' ( \rho_\Omega) = - \infty$ ensure respectively that 
any $C ^2$ function $\varphi$ touching $v$ from above  at $x_0$  satisfies $-\Delta _\infty \varphi (x_0) = 0\leq 1$, and that the class of $C^2$ functions touching $v$ at $x_0$ from below is empty. Thus $v$ satisfies the definition of viscosity solution also at $x_0$.   \qed

\bigskip

{\bf Proof of Theorem \ref{t:a}}.

Let us prove that the function  $\phi_{\Omega}$ is a viscosity
solution to problem \eqref{f:dirich}.  

Clearly, $\phi _\Omega$ satisfies the null Dirichlet condition on the boundary. 

We have to show that $\phi_{\Omega}$ satisfies the definition of viscosity solution
to the pde in \eqref{f:dirich} at every point
 $x_0\in\Omega$. 
To that aim, we distinguish the two cases $x_0 \in \Cut (\Omega)$ and $x_0 \not \in \Cut (\Omega)$. 

For simplicity of notation, since no ambiguity may arise, in the remaining the proof we denote by $d:= d _{\partial \Omega}$ the distance function from 
$\partial \Omega$.

\medskip
{\it Case $x_0\in\Cut(\Omega)$}.  

We check first that $\phi_{\Omega}$ 
satisfies \eqref{f:subsol}
at $x_0$.  
We notice that, 
if $v$  is defined by (\ref{defv}), the functions $\phi_\Omega$ and $v$ satisfy
\begin{equation}\label{f:comparison}
\phi _\Omega (x_0) = v( x_0)  \qquad \hbox{ and } \qquad \phi _\Omega (x) \geq v  (x) \quad \forall x \in \overline \Omega \setminus \{ x_0 \}\,. 
\end{equation}
Namely, the first condition is due to the equality $d (x_0) = \rho_\Omega$ (recall that by hypothesis $\Cut (\Omega) = \high (\Omega)$), and the second one follows
by using the monotonicity of $g$ and the triangular inequality $d (x) \geq \rho_\Omega - |x- x_0| = d (x_0) - |x-x_0|$.  
 
In view of (\ref{f:comparison}), 
if  $\varphi$ is a $C^2$ function
touching $\phi_{\Omega}$ at $x_0$ from above, then $\varphi$ touches also $v$ from above. 
Since from Lemma \ref{l:primo} we know that $-\Delta_{\infty}v= 1$ in $\Omega$, we infer that 
$
-\Delta_{\infty}\varphi(x_0) \leq 1$, and hence that $\phi _\Omega$ 
satisfies \eqref{f:subsol}
at $x_0$. 

The proof  that $\phi_{\Omega}$ 
satisfies also \eqref{f:supersol}
at $x_0$ is straightforward. Indeed, since
$\Cut (\Omega) = M (\Omega)$ and $g'' (\rho_\Omega) = - \infty$, the class of $C^2$ functions touching
$\phi_{\Omega}$ at $x_0$ from below turns out to be empty.

\medskip
{\it Case $x_0\not\in\Cut(\Omega)$}.  

Let $p_0\in\Cut(\Omega)$ and $q_0 \in \pi _{\partial \Omega} (x_0)$ be such that
$x_0 \in ]p_0, q_0[$ (the open segment joining $p_0$ and $q_0$), with $|q_0 - x_0| = d(x_0) $ and $|p_0-q_0| = \rho_\Omega$
(where the latter equality holds true by the assumption $\Cut (\Omega) = M (\Omega)$). 
Set
$$\nu := \frac{p_0-q_0}{|p_0-q_0|} = \nabla d (x_0)\,, $$
and let  $\varphi$ be a $C^2$ function
touching $\phi_{\Omega}$ at $x_0$ from above. 
Let us compare the two functions of one real variable defined for $t \in [0, \rho_\Omega] $ by
\begin{equation}\label{f:onevar}
h(t) := \varphi(y_0 + t\nu)  \qquad \hbox{ and  } \qquad g (t ) = \phi _\Omega ( y _0+ t \nu)\, . 
\end{equation}

Taking into account that $\varphi$ touches $\phi_{\Omega}$ at $x_0$ from above, and that $\phi _\Omega$ is differentiable at $x_0$, the functions $h$ and $g$ satisfy:
$$\begin{array}{ll}
& h ( d(x_0) ) = \varphi ( x_0) = \phi _\Omega (x_0) = g ( d (x_0)) \\ \noalign{\medskip}
& h' ( d(x_0) ) \nabla d ( x_0) = \nabla \varphi ( x_0) = \nabla \phi _\Omega (x_0) = g' ( d (x_0)) \nabla d ( x_0) \\ \noalign{\medskip}
& h (t) \geq g (t) \qquad \hbox{ near } t = d (x_0)\,.
\end{array}
$$
We infer that
\[
h''(d(x_0))  \geq g''(d(x_0)) \, ,  
\]
and in turn, recalling also that $g' ( d(x_0)) = h' ( d(x_0)) >0$,  that
\begin{equation}\label{miste}
(h' (d(x_0)) ^ 2 h''(d(x_0)) \geq (g' (d(x_0)) ^ 2 g''(d(x_0))\, . 
\end{equation}
Now, by definition the infinity-Laplacian of $\varphi$ is given by
\[
\Delta_{\infty} \varphi(x_0)  =
 \pscal{D^2\varphi(x_0) \nabla \varphi (x_0)}{\nabla \varphi (x_0)}= 
 h'(d(x_0))^2 h''(d(x_0))  \,.
\]
Then, by using (\ref{miste}) we obtain
\[
-\Delta_{\infty} \varphi(x_0) \leq -  (g' (d(x_0)) ^ 2 g''(d(x_0)) = 1\ ,
\]
where the last equality readily follows from the definition of $g$. Hence
$\phi_{\Omega}$ 
satisfies \eqref{f:subsol}
at $x_0$. 
The proof that $\phi_{\Omega}$ 
satisfies also \eqref{f:supersol}
at $x_0$ is completely analogous: if $\varphi$ is a $C^2$ function 
touching $\phi _\Omega$ at $x_0$ from below, it is enough to compare the two functions of one real variable 
defined by (\ref{f:onevar}), and argue in a similar way as above.

\smallskip 
We conclude that
$\phi_{\Omega}$ 
is a viscosity solution to the PDE in \eqref{f:dirich}. 
Since, by \cite[Thm.~5]{LuWang}, the viscosity solution to the Dirichlet boundary value problem
\eqref{f:dirich} is unique, the assertion that  $\phi_{\Omega}$ is the unique viscosity
solution follows.

\smallskip
Finally we observe that, 
if $\Omega$ is of class $C^1$, then the distance function $d$ is differentiable
also on $\partial\Omega$, and
$$\nabla\phi_{\Omega}(y) = g'(0)\nu(y) = ( 3 \rho _\Omega) ^ {1/3} \nu (y) \qquad \forall \, y\in\partial\Omega \,.$$
Therefore, $\phi _\Omega$ solves also the overdetermined boundary value 
problem (\ref{overdet}), with the value of the constant $c$ equal to  $( 3\rho _\Omega) ^ {1/3}$. 
\qed

\section{A geometric result on the distance function}\label{secgeo}

In this section we prove a new estimate on the distance function near singular points, which will be used as a crucial tool
to construct suitable viscosity test functions for problem (\ref{f:dirich}). 

We recall that the {\it Fr\'echet super-differential} of 
a function $u\in C^0(\Omega)$ at $x\in\Omega$ is defined by
\[
D^+ u (x) := 
\left\{ p \in \R ^n \ :\ \limsup _{ y \to x} 
\frac{u(y) - u(x) -\langle p, y-x \rangle }{|y-x| }  
\leq 0  \right\} \,.
\]
If $u$ is a locally Lipschitz function in $\Omega$,
the set $D^* u (x)$ of {\it reachable gradients} of $u$ at $x\in\Omega$ is the set of vectors 
$p \in \R^n$ for which
there exists a sequence $\{ x _h \} \subset \Omega \setminus \{ x \}$, 
with $u$ differentiable at $x_h$, such that
\[
\lim _h x _h = x \qquad \hbox{  and }  \qquad \lim _n \nabla u (x_h ) = p\,.
\]

Then the result reads:

\begin{theorem}\label{p:estid}
Let $\Omega\subset\R^n$ be an open set,
and let $x_0\in\Sigma(\Omega)$. 
Then for every $p\in D^+ d_{\partial \Omega}(x_0)\cap B_1(0)$
there exist
a constant $K>0$ and a unit vector
$\zeta\in\R^n$ 
satisfying the following property:
\begin{equation}\label{f:estidist0}
d_{\partial \Omega} (x) \leq d_{\partial \Omega} (x_0) + \pscal{p}{x-x_0}
- K\, |\pscal{\zeta}{x-x_0}| +\frac{1}{2 d_{\partial \Omega}(x_0)}
|x-x_0|^2,
\quad \forall x\in\Omega.
\end{equation}
In particular for every $c>0$ 
the inequality
\begin{equation}\label{f:estidist}
d_{\partial \Omega} (x) \leq d_{\partial \Omega} (x_0) + \pscal{p}{x-x_0}
- c \pscal{\zeta}{x-x_0}^2 +\frac{1}{2 d_{\partial \Omega}(x_0)}
|x-x_0|^2
\end{equation}
holds for every $x\in B_{\delta}(x_0)\cap\Omega$
with $\delta = K/c$.
Furthermore, if $p\neq 0$ then the vector $\zeta$ can
be chosen so that $\pscal{\zeta}{p}\neq 0$.
\end{theorem}

\begin{remark}
We shall see in a moment (cf.\ Proposition \ref{CS}) that, since $x_0\in\Sigma(\Omega)$, the set
$D^+ d_{\partial \Omega}(x_0)\cap B_1(0)$
is not empty and contains non-zero elements. Moreover, every 
$p\in D^+ d_{\partial \Omega}(x_0)\cap B_1(0)$ can 
be written as a convex combination of points
$p_0, \ldots p_k \in D^+ d_{\partial \Omega}(x_0)\cap \partial B_1(0)$, with $k \leq n$. 
We shall show, during the proof of Theorem~\ref{p:estid},
that the constant $K$ appearing in \eqref{f:estidist0}
can be chosen as the distance
between the origin and the boundary of the set
$\conv\{p_0-p,\ldots,p_k-p\}$,
whereas the vector $\zeta$ 
can be chosen in the set
\begin{equation}\label{f:zeta}
Z:= \left\{\frac{z}{|z|}\, :\
z\in\conv\{p_0-p,\ldots,p_k-p\},\ z\neq 0
\right\}\,.
\end{equation}
\end{remark}

\begin{remark}\label{r:reach}
If $\Omega$ is a set of positive reach, 
i.e., if there exists $R>0$ such that every point of
$\{z\in\R^n\setminus\overline{\Omega}:\ d_{\overline{\Omega}}(z) < R\}$
has a unique projection on $\overline{\Omega}$,
then estimate \eqref{f:estidist0}
can be improved in the following way:
\begin{equation}\label{f:estidistpr}
d_{\partial \Omega} (x) \leq d_{\partial \Omega} (x_0) + \pscal{p}{x-x_0}
- K\, |\pscal{\zeta}{x-x_0}| +\frac{1}{2 (d_{\partial \Omega}(x_0)+R)}
|x-x_0|^2\,
\end{equation}
(see Step 2 in the proof of Theorem \ref{p:estid}). 
In particular, if $\Omega$ is a convex set, then we get
\begin{equation}\label{f:estidistconv}
d_{\partial \Omega} (x) \leq d_{\partial \Omega} (x_0) + \pscal{p}{x-x_0}
- K\, |\pscal{\zeta}{x-x_0}| \,,
\qquad\forall x\in \Omega.
\end{equation}
Analogous improvements hold also for inequality \eqref{f:estidist}.
\end{remark}

\begin{remark}
By writing (\ref{f:estidist}) at $x = x_0 \pm h$ and summing up, one arrives at
\[
d _{\partial \Omega} (x_0 +  h) + d _{\partial \Omega} (x_0 -  h) \leq 2 d _{\partial \Omega} (x_0) - 2c \pscal{\zeta}{h} ^ 2 + \frac{1}{d (x_0)} |h| ^ 2\,. 
\]
This shows that Theorem \ref{p:estid} can be seen as a refinement at singular points of the concavity estimate for the distance function given in \cite[Prop.\ 2.2.2]{CaSi}. 
\end{remark}

The remaining of this section is devoted to prove Theorem \ref{p:estid}, and then to exemplify it in the cases of two simple geometries, such as a square and a triangle.

\bigskip
Let us recall some definitions and known results on the non-smooth analysis of the distance function. 
Let $d_S (x) := \min _{y \in S} |x-y|$ denote the distance from a nonempty closed set $S$.  
By regular (resp.\ singular) points of $d_S$ we mean points where $d_S$ is differentiable (resp.\ not differentiable).  
Moreover, we denote by 
$\pi_S (x)= \{ y \in S \, :\, d _ S(x) = |x-y| \} $ 
the projection of a point $x$ onto $S$. 


\begin{proposition}\label{CS} 
Let $S$ be a nonempty closed set in $\R ^n$
and let $x\not\in S$. Then the following hold:
\begin{itemize} 
\item[(i)] $x$ is a regular point of $d_S$ 
if and only if $\pi _ S(x)$ is a singleton, and in this case there holds
$$\nabla d_S (x) = \frac{ x - \pi _S(x) } {d_S(x)}\,.$$ 
\item[(ii)] If $x$ is a singular point of $d_S$, there holds  
\begin{eqnarray}
&D ^ * d _S(x) = \displaystyle{\Big \{ \frac{x-y}{|x-y|} \ :\ y \in \pi _ S (x) \Big \} } & \label{D*} \\
\noalign{\medskip}
&D ^ + d_ S  (x) = \conv D^* d_S (x) 
= \dfrac {x - \conv (\pi _S(x))}{d_S(x)}  \,.  & \label{D+}
\end{eqnarray}
\item[(iii)] The set ${\rm extr}\, D^+ d_S(x)$ 
of extremal points of $D^+ d_S(x)$ is given by
$${\rm extr}\, D^+ d_S(x) = D^* d_S(x) = D^+ d_S(x) \cap \partial B_1(0)\,.$$
\end{itemize}
\end{proposition}

\proof 
The properties (i) and (ii) follow from Corollary 3.4.5 in \cite{CaSi}.
To prove (iii) notice firstly that, by \eqref{D*}--\eqref{D+},
it holds ${\rm extr}\, D^+ d_S(x) \subseteq D^* d_S(x) \subset \partial B_1(0)$.
On the other hand, since $D^+ d_S(x) \subseteq \overline{B}_1(0)$,
if $p\in D^+d_S(x) \cap \partial B_1(0)$ then necessarily
$p\in {\rm extr} \, D^+d_S (x)$.
\qed 

\begin{remark}\label{l:hull}
As a consequence of Proposition \ref{CS} (ii)-(iii),
we have that,
if $\Omega$ is an open subset of $\R^n$
and $x\in \high (\Omega)$,
then 
\[
x \in \conv \big ( \partial \Omega \cap \partial B_{\rho_\Omega}(x)\big )\,.
\]
Namely, since $x\in\high(\Omega)$, then $0\in D^+ d_{\partial\Omega}(x)$,
so that there exist $p_0, \ldots, p_k \in D^* d_{\partial\Omega}(x)$ and
numbers $\lambda_0, \ldots, \lambda_k \in [0,1]$ with $\sum_{i=0}^k \lambda_i = 1$
such that $\sum_{i=0}^k \lambda_i p_i = 0$.
On the other hand, for every $i$ there exists 
$y_i\in\pi_{\partial\Omega}(x)$ such that $p_i = (x-y_i) / \rho_{\Omega}$, 
so that
\[
0 = \rho_{\Omega} \sum_{i=0}^k \lambda_i p_i = \sum_{i=0}^k \lambda_i (x - y_i)
\]
from which we conclude that $x = \sum_{i=0}^k \lambda_i y_i$.
\end{remark}



\bigskip

We are now in a position to give:

\bigskip
\textbf{Proof of Theorem \ref{p:estid}.}
Set for brevity $d: = d _{\partial \Omega}$. 
Since $p\in D^+ d(x_0)\setminus D^* d(x_0)$
and $D^+d(x_0) = \conv D^* d(x_0)$,
there exist $p_0, \ldots, p_k \in D^*d(x_0)$
(with $1\leq k\leq n$)
and numbers 
$\lambda_0, \ldots, \lambda_k\in (0,1)$ with
$\sum _{i=0} ^ k \lambda _i = 1$ such that
$p = \sum_{i=0}^k \lambda_i p_i$.

We divide the remaining of the proof in three steps.

\medskip
\textsl{Step 1. 
The following inequality holds:}
\begin{equation}\label{f:estid12}
d(x) \leq d(x_0) + \min_{i=0,\ldots,k} \pscal{p_i}{x-x_0}
+\frac{1}{2 d(x_0)} |x-x_0|^2\,,
\qquad \forall x\in\Omega.
\end{equation}

\smallskip
Since the points $y_i := x_0 - d(x_0) p_i$, $i=0, \ldots, k$,
belong to $\pi_{\partial\Omega}(x_0)$,
and recalling that $\sqrt{1+t}\leq 1+t/2$ for every $t\geq -1$,
we have the estimate
\[
\begin{split}
d(x) & \leq |x-y_i| 
= |x-x_0 + d(x_0) p_i|
= \left[
|x-x_0|^2 + 2d(x_0) \pscal{p_i}{x-x_0} + d(x_0)^2
\right]^{1/2}
\\ & =
d(x_0) \left[1+\frac{2}{d(x_0)}\pscal{p_i}{x-x_0}
+\frac{1}{d(x_0)^2} |x-x_0|^2\right]^{1/2}
\\ & \leq
d(x_0) + \pscal{p_i}{x-x_0} 
+\frac{1}{2 d(x_0)} |x-x_0|^2\,.
\end{split}
\]
Since this inequality holds for every $i$,
\eqref{f:estid12} follows.
We observe that, if $\Omega$ is a set of positive reach
(see Remark \ref{r:reach}), then we can obtain the improved estimate
\begin{equation}\label{modif}
d(x) \leq
d(x_0) + \min _{i = 0, \dots, k }  \pscal{p_i}{x-x_0} 
+\frac{1}{2 (d(x_0)+R)} |x-x_0|^2\, ,
\end{equation}
by using, in place of $y _i$,  the points $\widetilde{y}_i := x_0 - (d(x_0) + R) p_i$
and the fact that $B_R(\widetilde{y}_i)\cap\Omega = \emptyset$.
{}If in Steps 2 and 3 below we use the improved estimate (\ref{modif}) in place of (\ref{f:estid12}), 
we can then obtain 
\eqref{f:estidistpr} instead of \eqref{f:estidist0}.

\medskip
\textsl{Step 2. 
Let $K$ denote the distance between the origin and the boundary
of $\conv\{p_0-p,\ldots,p_k-p\}$.
Then for every unit vector $\zeta$ in the set $Z$
defined in \eqref{f:zeta}, one has
\begin{equation}\label{f:estimin22}
\min_{i=0,\ldots,k} \pscal{p_i-p}{x}\leq
-K\, |\pscal{\zeta}{x}|\,,
\qquad \forall x\in\R^n.
\end{equation}
}

\smallskip
Since $\sum_{i} \lambda_i (p_i - p) = 0$
we have that the set
\[
F := {\rm span}\{p_0-p, p_1-p, \ldots, p_k-p\}
\]
is a subspace of $\R^n$ of dimension $k$.
Let
\[
Q := \conv\{p-p_0, p-p_1,\ldots, p-p_k\}\,;
\]
since $0$ belongs to the relative interior of the polytope $Q$,
and since $K$ is the distance between $0$ and
the boundary of $Q$, we clearly have $K > 0$
and $B := \overline{B}_K(0)\cap F \subseteq Q$.
Hence
\[
h_Q(x) := \max\{\pscal{q}{x}:\ q\in Q\}
\geq
\max\{\pscal{b}{x}:\ b\in B\} =: h_B(x),
\qquad \forall x\in\R^n.
\]
On the other hand, we have that
\[
h_Q(x) = \max_{i=0,\ldots k} \pscal{p-p_i}{x}
= - \min_{i=0,\ldots k} \pscal{p_i-p}{x}
\]
whereas, if $\zeta$ is any unit vector in the set $Z$ defined in \eqref{f:zeta},
then $\pm K\zeta\in B$, so that
\[
h_B(x) = \max\{\pscal{b}{x}:\ b\in B\}
\geq K |\pscal{\zeta}{x}|\,.
\]
Now \eqref{f:estimin22} 
easily follows.

\medskip
\textsl{Step 3. Completion of the proof.}

\smallskip
The estimate \eqref{f:estidist0} is a direct consequence
of \eqref{f:estid12} and \eqref{f:estimin22}.
In order to prove \eqref{f:estidist}
it is enough to observe that,
given $c > 0$,
the inequality
$K \, |t| \geq c\, t^2$
holds for every  $|t| < K/c$.
\qed

\begin{figure}[ht]
\begin{minipage}{0.45\linewidth}
\centering
\def\svgwidth{6cm}   
{
  \setlength{\unitlength}{\svgwidth}
  \begin{picture}(1,1.04308078)%
    \put(0,0){\includegraphics[width=\unitlength]{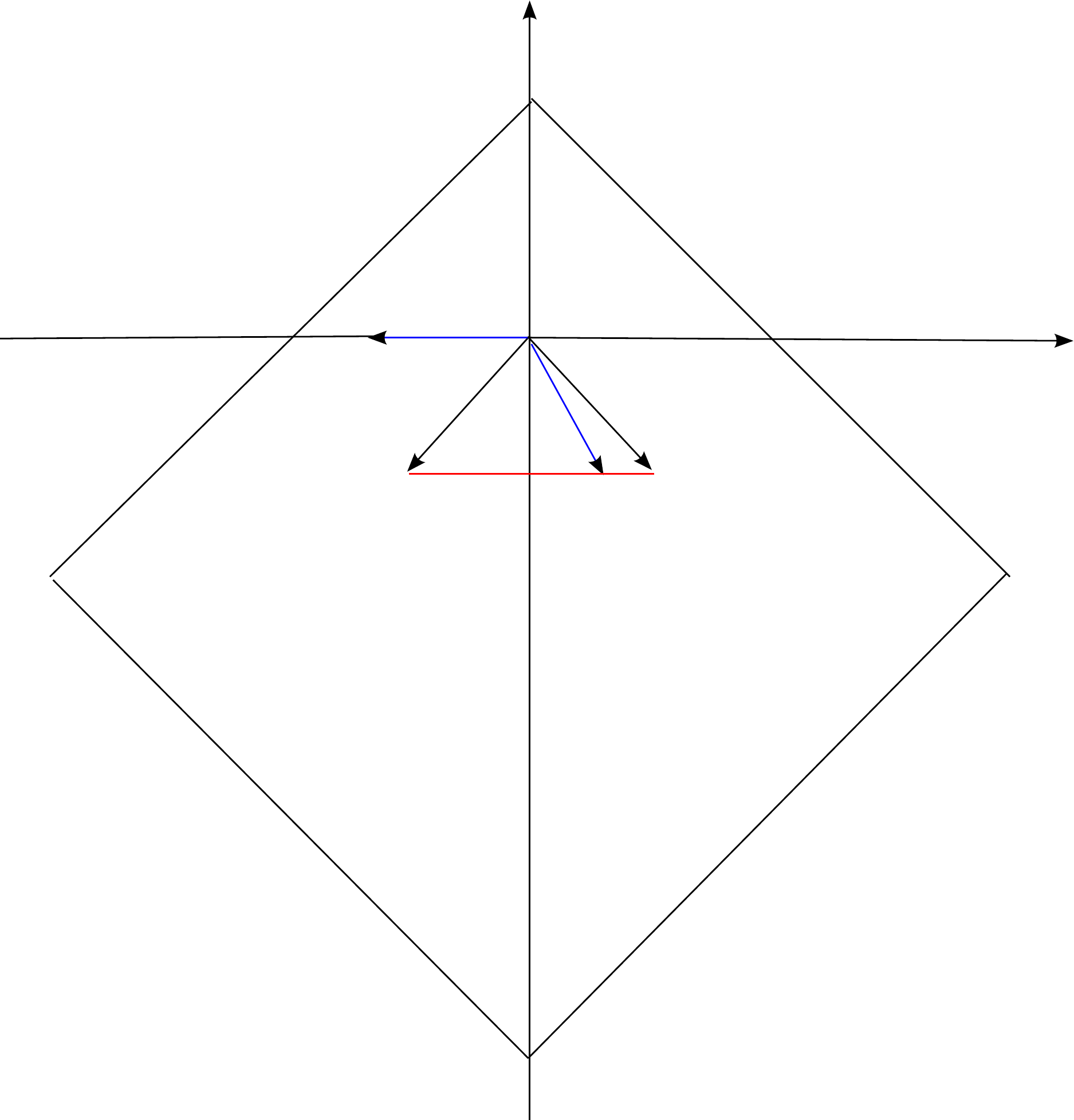}}%
    \put(0.35688182,0.5689513){\color[rgb]{0,0,0}\makebox(0,0)[lb]{\smash{$p_0$}}}%
    \put(0.59148569,0.56648424){\color[rgb]{0,0,0}\makebox(0,0)[lb]{\smash{$p_1$}}}%
  \end{picture}%
}%
\end{minipage}
\begin{minipage}{0.45\linewidth}
\centering
\def\svgwidth{6cm}   
{
  \setlength{\unitlength}{\svgwidth}
  \begin{picture}(1,0.77412266)%
    \put(0,0){\includegraphics[width=\unitlength]{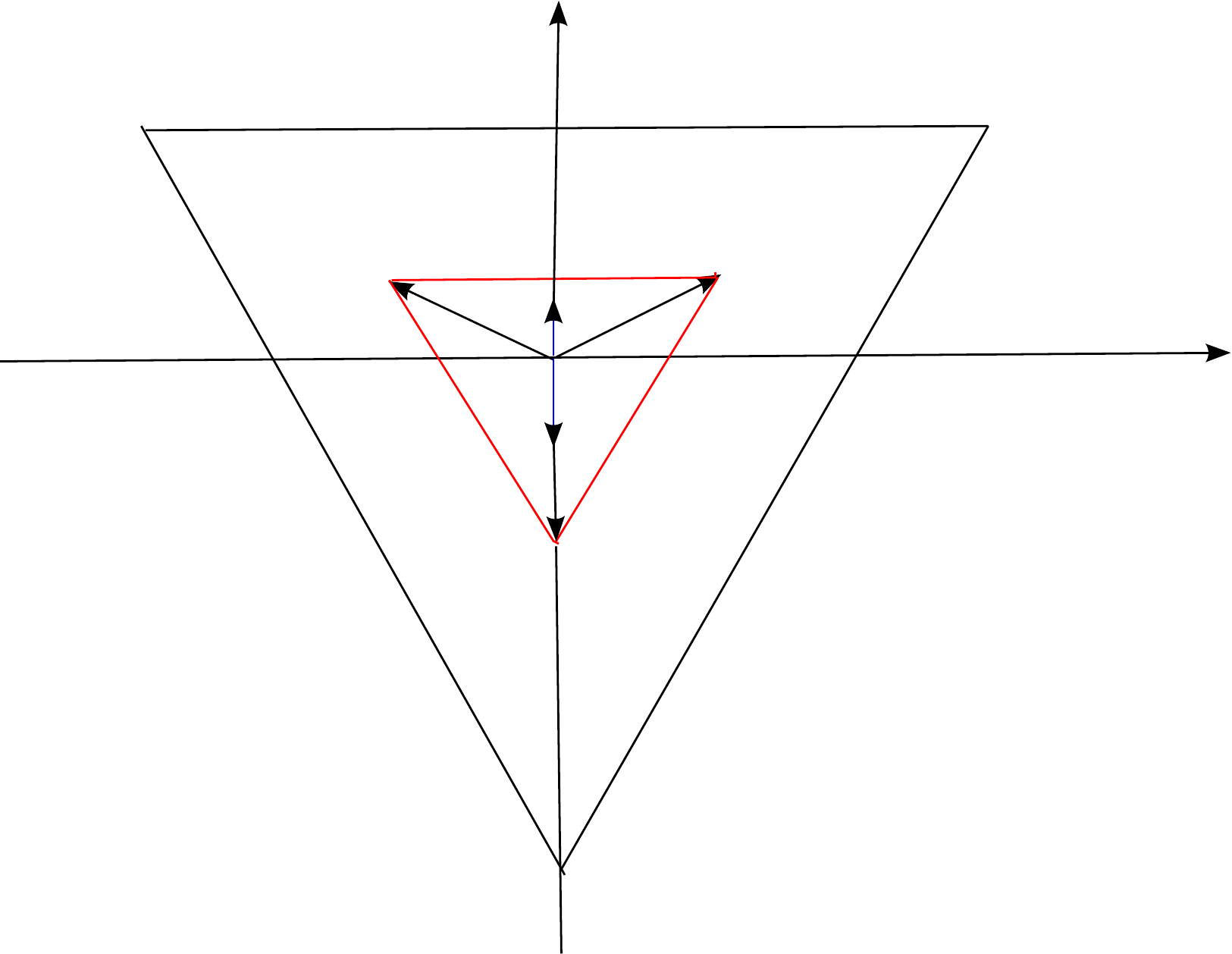}}%
    \put(0.45350966,0.3104516){\color[rgb]{0,0,0}\makebox(0,0)[lb]{\smash{$p_0$}}}%
    \put(0.57213118,0.5728885){\color[rgb]{0,0,0}\makebox(0,0)[lb]{\smash{$p_1$}}}%
    \put(0.2824008,0.56658998){\color[rgb]{0,0,0}\makebox(0,0)[lb]{\smash{$p_2$}}}%
  \end{picture}%
}%
\end{minipage}
\caption{Application of Theorem \ref{p:estid} to a square and a triangle}
\label{Ex:quadrato}
\end{figure}
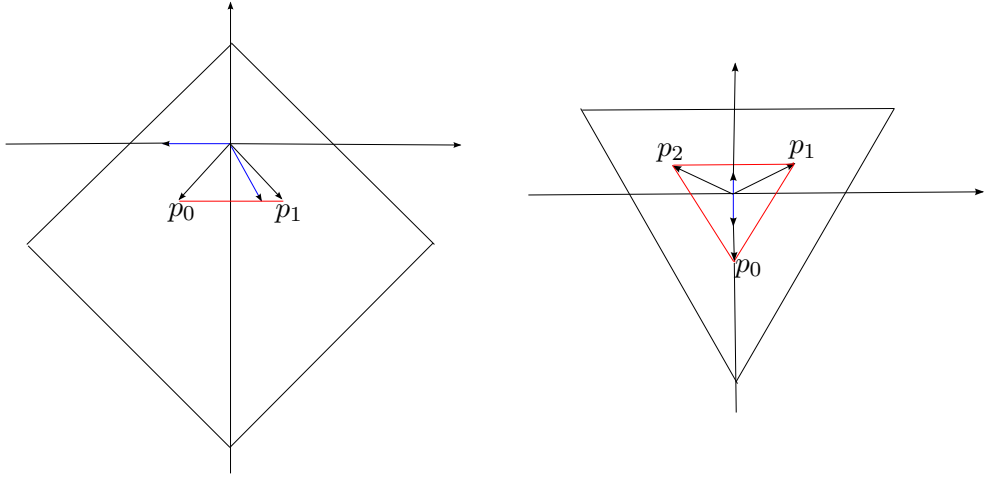

\begin{example} 
Let $Q$ be the square with sides of length $4$, having two vertices at 
$(0, \sqrt 2)$ and $(0 , - 3 \sqrt 2)$.  
Let us show how the construction of Theorem \ref{p:estid} applies at the point 
$x_0 = (0, 0) \in \Sigma (Q)$, see Figure \ref{Ex:quadrato} left.  
Adopting the same notation as in the above proof, there holds 
\[
D ^* d_{\partial Q} (x_0) = \{ p _0, p _1\} \, , \hbox{ with } p_0 = \big ( - \frac{1}{\sqrt 2}, - \frac{1}{\sqrt 2} \big ) \hbox{  and  }
p_1 = \big  (  \frac{1}{\sqrt 2}, - \frac{1}{\sqrt 2} \big )\, .
\] 
Choosing 
\[
p= \frac{1}{2} \big ( p _0 +  p _1 \big )= 
\left(0  ,   - \frac{1}{\sqrt 2} \right) \in D ^+ d _{\partial Q} (x_0)
\]
and 
\[
\zeta = \frac{p_0 - p}{|p_0 - p|} = (-1, 0),
\] 
we have that $K = \min\{|p-p_0|,\, |p-p_1|\} = 1/\sqrt{2}$,
and the estimate \eqref{f:estidistconv} (for convex $\Omega$)
takes the form
\[
d_{\partial Q}(x) \leq 1 - \frac{1}{\sqrt{2}}\, x_2 - \frac{1}{\sqrt{2}}\, |x_1|.
\]
A direct computation of $d_{\partial Q}$ shows that, indeed,
the inequality above turns out to be an equality on the
upper half--square.
\end{example}

\begin{example}
Let $T$ be triangle with sides of length $2$, 
having vertices at 
$\big (1, \frac{\sqrt 3}{3} \big )$, $\big (-1, \frac{\sqrt 3}{3} \big )$
and $\big(0, -2\frac{\sqrt{3}}{3}\big)$. 
Let us show how the construction of Theorem \ref{p:estid} applies at the point $x_0 = (0, 0)$; notice that in this case there holds 
$x _0 \in \high (T)$, as $d_{\partial T} (x_0) = \frac{\sqrt 3}{3}$,  
see Figure \ref{Ex:quadrato} right. There holds 
\[
D ^* d_{\partial T} (x_0) = \{ p _0, p _1, p_2\} \, , \hbox{ with } p_0 = (0, -1), p _1 = \left(  \frac{\sqrt 3}{2},  \frac{1}{2} \right) \hbox{  and  }
p_2 = \left( - \frac{\sqrt 3}{ 2},  \frac{1}{ 2}  \right)\, .
\]
Choosing
\[
p = (0,0) = \frac{1}{3} \big ( p _0 + p _1 + p _2 \big ) \in D ^+ d _{\partial T} (x_0)\,,
\qquad
\zeta = \frac{p_0 - p}{|p_0 - p|} = ( 0, -1)\, ,
\]
the resulting estimate (for convex sets) reads 
\[
d _{\partial T} (x) \leq 
\frac{\sqrt{3}}{3} - \frac{1}{2} \, |x_2|\,.
\]  
This estimate can be compared with the exact value of 
$d _{\partial T}$:
\[
d _{\partial T}(x) = \frac{\sqrt{3}}{3} + 
\min\left\{-\frac{\sqrt{3}}{2}\, |x_1| + \frac{x_2}{2}\,,
-x_2\right\}\,.
\]
\end{example}


\section{Necessity of the condition $\Cut (\Omega) = \high (\Omega)$}\label{secb}

In this section we establish the converse statement of Theorem \ref{t:a}, which reads: 

\begin{theorem}
\label{t:b}
Let $\Omega \subset \R^n$ be an open bounded connected domain. 
Assume there exists a web viscosity solution $u$ to the Dirichlet boundary value problem $(\ref{f:dirich})$. 
Then it holds $u = \phi _\Omega$ and 
$\Cut (\Omega) = \high (\Omega)$.
\end{theorem}

\medskip
By combining Theorem \ref{t:b} with the geometrical results we proved in  \cite[Thm.\ 6 and Thm.\ 12]{CFb}, one readily gets the following two corollaries. 

The first one  allows to view the shape of domains where problem (\ref{f:dirich}) admits a web viscosity solution:

\begin{corollary}\label{corgeo1} 

Under the same hypotheses of Theorem \ref{t:b}, assume in addition that $n=2$. 
Then the set  $S:=\Cut (\Omega) = \high (\Omega)$ 
is either a singleton or a $1$-dimensional manifold of class $C ^ {1,1}$, and $\Omega$ is the tubular neighborhood
$$\Omega=  S_{\rho_\Omega} := \{ x \in \R ^2 \ :\ d_S (x) < \rho_\Omega \}\,.$$
\end{corollary}

The second one allow to establish symmetry under suitable assumptions on the space dimension and on the topology of $\Omega$:

\begin{corollary}\label{corgeo2}
Under the same hypotheses of Theorem  \ref{t:b},  assume in addition that $\Omega$ is of class $C ^2$. If either 
$n = 2$ and $\Omega$ is simply connected, or $n$ is arbitrary and 
$\Omega$ is convex, then $\Omega$ is a ball. \end{corollary}

\begin{remark}  If  the conclusion $\Cut (\Omega) = \high (\Omega)$ of Theorem \ref{t:b} should remain valid when one assumes merely the  existence of a solution to problem (\ref{overdet}), in view of Corollaries \ref{corgeo1} and \ref{corgeo2} a crucial difference would emerge between the overdetermined boundary value problems (\ref{pLap}) for the classical Laplacian and (\ref{overdet}) for the $\infty$-Laplacian: while for the former symmetry holds under very mild conditions on $\partial \Omega$ \cite{Praj, Vog}, 
for the latter symmetry would be still true if 
$\Omega$ is simply connected and
$\partial \Omega \in C^2$, but false below such threshold of regularity. 
\end{remark}

\medskip
Let us give some preliminary results needed for the proof of Theorem \ref{t:b}. 

Below,  $D^{\pm} f (\rho- |z_0|)$ denote  the Fr\'echet super and sub-differentials of $f$ at $\rho - | z_0|$. 

\medskip
\begin{lemma}
\label{l:subdif}
Let
$f\colon [0,\rho]\to\R$ be a continuous function and, for $z\in B_\rho(0)$,  set
$v(z) := f(\rho-|z|)$. Let $z_0 \in B _\rho (0) \setminus \{ 0 \}$ be fixed and assume that $\psi$ is a $C ^1$ function. 
\begin{itemize}
\item[(a)] If $\psi$ touches $v$ from above at   $z_0$, 
then $D^+ f(\rho- |z_0|)\neq \emptyset$ and
\[
\nabla\psi(z_0) = \alpha\, \frac{z_0}{|z_0|}\,,
\quad \hbox{ with } \ \alpha\in - D^+ f(\rho- |z_0|).
\]

\item[(b)] If $\psi$ touches $v$ from below at   $z_0$, 
 then $D^- f(\rho - |z_0|)\neq \emptyset$ and
\[
\nabla\psi(z_0) = \alpha\, \frac{z_0}{|z_0|}\,,
\quad \hbox{ with } \ \alpha\in - D^- f(\rho- |z_0|).
\]
\end{itemize}

\end{lemma}

\begin{proof}
We prove only statement (a), as the proof of (b) is completely analogous. 

Let $\zeta _0:= z_0 /|z_0|$ and define the map $z: [ 0 , \rho] \to B _\rho (0) $ by 
$z(s) := (\rho-s)\zeta_0$.

Setting $s_0 := \rho- |z_0|$, we have \[
\begin{split}
f(s) & = f(\rho - |z(s)|) = v(z(s)) 
\\ & \leq
\psi(z(s)) = \psi(z_0) + \pscal{\nabla\psi(z_0)}{z(s)-z_0} +
o(|z(s) - z_0|)
\\ & =
f(s_0) - (s-s_0) \pscal{\nabla\psi(z_0)}{\zeta_0} + o(|s-s_0|).
\end{split}
\]
This shows that, setting $\alpha :=\pscal{\nabla\psi(z_0)}{\zeta_0}$, it holds $-\alpha\in D^+ f(s_0)$.

Let now $n_0$ be a unit vector orthogonal to $\zeta_0$, and consider an
arc of circumference $\gamma(s)$, $s \in ( - \epsilon, \epsilon)$, such that $\gamma(0) = z_0$,
$\gamma'(0) = n_0$ and $|\gamma(s)| = |z_0|$.
We have 
\[
\begin{split}
\psi(z_0)  &= f(\rho- |z_0|) = f(\rho- |\gamma(s)|) =
v(\gamma(s)) 
\\ & \leq \psi(\gamma(s)) =
\psi(z_0) + \pscal{\nabla\psi(z_0)}{n_0}  \, s + o(s)\,.
\end{split}
\]
Then  it must be $\pscal{\nabla\psi(z_0)}{n_0}= 0$, completing the proof.
\end{proof}

\bigskip
\begin{proposition}
\label{p:caratt} Let $\Omega\subset\R^n$ be a non-empty bounded open set, 
let $f\colon [0,\rho _\Omega]\to\R$ be a continuous function, and assume that
$u(x) := f(d_{\partial \Omega}(x))$ is a web viscosity solution
of
\begin{equation}
\label{f:equau}
-\Delta_{\infty} u = 1 \qquad \text{in}\ \Omega.
\end{equation}
Then:
\begin{itemize}
\item[(i)] the map $t \mapsto f (t)$ is monotone increasing on $[0, \rho _\Omega]$;

\smallskip
\item[(ii)] the function $v(z) := f(\rho _\Omega -|z|)$ is a viscosity solution of
\begin{equation}
\label{f:visrad}
-\Delta_{\infty} v = 1 \qquad \text{in}\ 
B_{\rho_\Omega}(0)\setminus\{0\}.
\end{equation}

\end{itemize}
\end{proposition}

\begin{proof}
(i) Assume by contradiction that $t \mapsto f (t)$ is not monotone increasing  on $[0, \rho _\Omega]$: let $t_1, t_2 \in [0, \rho _\Omega]$ be such that  $t_1 < t _2$ but $f ( t _1) > f ( t_2)$. Then the absolute minimum of the continuous function $f$ on the interval $[t_1, \rho _\Omega]$ is attained at some point $t_0 > t _1$; in particular, there exists a point $t_0 \in (0, \rho _\Omega]$ which is of local minimum for the map $f$. Let us show that this  fact is not compatible with the assumption that  $u (x)= f ( d _{\partial \Omega} (x))$ is a web viscosity solution to (\ref{f:equau}).  Since $t_0>0$, there exists a point  $x_0$ lying  in $\Omega$ such that $d _{\partial \Omega} (x_0) = t _0$. Since $t_0$ is a local minimum for the map $f$, the point $x_0$ is a local minimum for the function $u$. Then, we can construct a $C ^2$ function $\varphi$ which touches $u$ from below at $x_0$, and is locally constant in a neighborhood of $x_0$, namely $\varphi (x) = u (x_0)$ for every $x \in B _ r (x_0)$ (for some $r>0$). Clearly it holds $-\Delta _\infty \varphi  = 0 < 1$, against the fact that $u$ is a viscosity super-solution. 

\medskip
(ii) Let $z_0\in B_{\rho_\Omega}(0)\setminus\{0\}$ be fixed. Let us prove that $v$ is a viscosity sub-solution to (\ref{f:visrad}) at $z_0$.  If $\psi$ is
a $C^2$ function touching $v$ from above at $z_0$, 
we have to show that
\begin{equation}
\label{f:sspsi}
- \Delta_{\infty}\psi(z_0) =
-\pscal{D^2\psi(z_0) \nabla\psi(z_0)}{\nabla\psi(z_0)} \leq 1.
\end{equation}
We choose  a maximal ray $[p_0, q_0]$, with $p_0 \in M (\Omega)$ and $q_0 \in \partial \Omega$, that is, $p_0$ is the center of a ball of radius $\rho _\Omega = |p _0 - q _0|$ contained into $\Omega$. 
We pick a point $x_0  \in \Omega$ such that  
$$x_0 \in ] p _0, q _0[ \qquad \hbox{ and } \qquad d _{\partial \Omega} (x_0) = \rho _\Omega - | z_0|$$ and, for $x$ belonging to a neighborhood of $x_0$,  we set
$$
z (x) := \big [ \rho _\Omega - | x - q _0 | \big ] \zeta _0\ ,\qquad \hbox{ with } \zeta_0 := \frac{z_0} { | z_0 |}\,.
$$
In particular, notice that by construction there holds $z (x_0) = z _0$. 

We now consider the composite map
$$\varphi (x) := \psi ( z ( x))\, ,$$ 
which is clearly of class $C ^2$ in a neighborhood of $x_0$. We claim that $\varphi$ touches $u$ from above at $x_0$. Indeed, by  the definitions of $u, v$, and $z$, and since $\psi$ touches $v$ from above at $z_0= z (x_0)$, there holds

\[
u(x_0) = f ( d _{\partial \Omega} (x_0) ) = f (\rho _\Omega - |z_0|) = v(z_0) = \psi(z_0) = \varphi(x_0). 
\]
Moreover there exists $r>0$ such that
\[ 
u(x) = f(d_{\partial \Omega}(x))  \leq f (\rho _\Omega - | z (x)| )  
= v(z(x)) \leq \psi(z(x)) = \varphi(x)
\qquad\forall x\in B_{r}(x_0). 
\]
Notice that the first inequality in the line above follows from statement (i) already proved
(taking into account that $| z(x)| \leq \rho _\Omega - d _{\partial \Omega} (x)$), while the second one holds for $r$ sufficiently small by the assumption that $\psi $ touches $v$ from above at $z_0$ and the continuity the map $z$ at $x_0$. 

Then, since $\varphi$ touches $u$ from above at $x_0$ and by assumption $u$ is a viscosity solution to \eqref{f:equau},
we deduce that \begin{equation}
\label{f:ssphi}
-\Delta_{\infty}\varphi(x_0) =
-\pscal{D^2\varphi(x_0) \nabla\varphi(x_0)}{\nabla\varphi(x_0)}  \leq 1.
\end{equation}
Setting $\delta(x) := | x - q_0|$, a direct computation yields
\[
\begin{split}
& \nabla\varphi(x) = - \pscal{\nabla\psi(z(x))}{\zeta_0}\, \nabla \delta(x),\\
& D^2\varphi(x) = \pscal{D^2\psi(z(x))\, \zeta_0}{\zeta_0}\,
\nabla \delta(x) \otimes \nabla \delta(x)
- \pscal{\nabla\psi(z(x))}{\zeta_0}\, D^2 \delta(x)\,.
\end{split}
\]
Taking into account the identities
\[
\begin{split}
& [\nabla \delta(x) \otimes \nabla \delta(x)] \nabla \delta(x) = \nabla \delta(x),\\
& D^2 \delta(x)\, \nabla \delta(x) = 0,
\end{split}
\]
we obtain
\begin{equation}\label{f:1}
\pscal{D^2\varphi(x_0) \nabla\varphi(x_0)}{\nabla\varphi(x_0)}  =
\pscal{D^2\psi(z_0)\, \zeta_0}{\zeta_0} \,
\left(\pscal{\nabla\psi(z_0)}{\zeta_0}\right)^2\,.
\end{equation}
Now, from Lemma~\ref{l:subdif} (a) we have
$$
\nabla\psi(z_0) = \alpha\zeta_0, \qquad \hbox{ with  } 
 \alpha\in -D^+f(
\rho _\Omega -|z_0|)\,.
$$
Therefore, \begin{equation}\label{f:2}
\pscal{D^2\psi(z_0)\, \zeta_0}{\zeta_0} \,
\left(\pscal{\nabla\psi(z_0)}{\zeta_0}\right)^2 = \pscal{D^2\psi(z_0)\, \nabla\psi(z_0)}{\nabla\psi(z_0)}\,.
\end{equation}
In view of (\ref{f:1}) and (\ref{f:2}), we conclude that \eqref{f:sspsi} follows from \eqref{f:ssphi}.

In order to prove that $v$ is a viscosity super-solution to (\ref{f:visrad}) at $z_0$, one can argue in a completely analogous way. More precisely, keeping the same definitions of $\zeta _0$, $p_0$, $q_0$, and $x_0$ as above, one has just to modify the auxiliary function $z(x)$ into
$\tilde z (x) := |x- p _0 | \zeta _0$, then replace the distance function $\delta (x)$  by $\tilde \delta (x) := |x - p _0|$, and finally
apply part (b) in place of part (a) of Lemma~\ref{l:subdif}. 
\end{proof}

\bigskip

{\bf Proof of Theorem \ref{t:b}}. 
Throughout the proof, we set for brevity $d:= d _{\partial \Omega}$.

Let us first prove the equality $u = \phi _\Omega$. 
Since by assumption $u$ is a web function and belongs to $C ^ 0 (\overline \Omega)$ 
(because a viscosity solution to \eqref{f:dirich} is by definition continuous up to the boundary), 
there exists a continuous function 
$f\colon [0,\rho_\Omega]\to\R$ such that 
$$u(x) = f(\dist(x))\ .$$
We have to show that  $f$ agrees with the function $g$ defined by (\ref{defg}). 

Since $u$ is assumed to be a viscosity solution to the Dirichlet problem (\ref{f:dirich}), by Proposition \ref{p:caratt} the function $v(z) := f(\rho _\Omega -|z|)$ is a viscosity solution to 
\begin{equation}
\label{f:dirisrad}
\begin{cases}
-\Delta_{\infty} v = 1 & \text{in}\ B_{\rho _\Omega}(0)\setminus\{0\},\\
v = 0 & \text{on}\ \partial B_{\rho _\Omega}(0)\\
v(0) = f (\rho _\Omega)\,.
\end{cases}
\end{equation}

Let us define, for every
$r >0$, the function 
\begin{equation}\label{defgc}
g_r(t) :=  c_0  \left [r ^ {4/3} - ( r - t ) ^ {4/3} \right ],
\qquad t\in [0, r]. 
\end{equation}

We claim that there exists $r \in [\rho _\Omega, + \infty)$ such that 
\begin{equation}\label{f:rgiusto}
g_{r} (\rho _\Omega) = f ( \rho _\Omega) \,.
\end{equation}

To prove this claim, we observe that the function 
$$r \mapsto  g _ r(\rho _\Omega)= c_0  \left [ r ^ {4/3} - ( r - \rho _\Omega ) ^ {4/3} \right ] $$ 
maps the interval $[\rho _\Omega, + \infty)$ onto $[c_0 \rho _\Omega ^ {4/3}, + \infty)$. 
Thus in order to show the existence of some $ r$ such that (\ref{f:rgiusto}) holds, it is enough to prove the inequality
\begin{equation}\label{f:forigin}
f (\rho _\Omega) \geq c_0 \rho _\Omega ^ {4/3}\,.
\end{equation}
In turn, this inequality readily follows by a comparison principle holding for the Dirichlet problem (\ref{f:dirich}). 
Namely, let $x_0 \in M (\Omega)$.  By Lemma \ref{l:primo}, the function $w(x):= g (\rho _\Omega - |x-x_0|)$  solves $- \Delta _\infty w = 1$ in $B _{\rho _\Omega}(x_0)$ and $w = 0$ on $\partial B _{\rho _\Omega}(x_0)$. On the other hand, the function $u$ solves $- \Delta _\infty u = 1$ 
in $B _{\rho _\Omega}(x_0)$ and $u \geq 0$ on $\partial B _{\rho _\Omega}(x_0)$. The latter inequality can be deduced by applying
the following result proved in \cite[Thm.\ 3]{LuWang}:
if $w_1 ,w_2  \in C (\overline A)$ are 
respectively a viscosity sub- and super-solution to
$-\Delta_\infty w = 1$ in $A$,
and $w_1 \leq w_2 $ on $\partial A$, then $w_1 \leq w_2$ in $A$. 

Again by applying the same result, we  deduce that $u (x) \geq g (\rho _\Omega - |x-x_0|)$ in $B _{\rho _\Omega}(x_0)$. This implies in particular 
$$f (\rho _\Omega)  = u (x_0) \geq g (\rho _\Omega) = c _0 \rho _\Omega ^ {4/3}\, $$
and concludes the proof of the claim. 

Now, we have that the function
\[
 g_{r}(\rho _\Omega -|z|), \qquad z\in B_{\rho _\Omega} (0),
\]
is a classical solution (and hence a viscosity solution)
to  problem \eqref{f:dirisrad}. (Notice that in particular the third equation in (\ref{f:dirisrad}) is satisfied thanks to (\ref{f:rgiusto})).

From \cite[Theorems 1 and 5]{LuWang}, we know that
there exists a unique viscosity solution to (\ref{f:dirisrad}).  We conclude that, for some $r \geq \rho _\Omega$, it holds
$v(z) = g_{r} (\rho _\Omega -|z|)$, that is 
\begin{equation}\label{fgc}
f (\rho _\Omega -|z|) = g_{r} (\rho _\Omega -|z|) \, ,
\end{equation}
or equivalently $u (x) = g_r ( d(x))$.

To conclude, we have to show the following equalities: 
$$u = \phi _\Omega \qquad \hbox{ and } \qquad \Cut (\Omega ) = \high (\Omega)\,.$$

\smallskip
\textsl{Proof of the equality $u = \phi _\Omega$}. 

\smallskip
Since we know that $u ( x) = g _ {r} ( d(x))$ for some $r \geq \rho _\Omega$, all we have to prove is that 
$r = \rho _\Omega$.   
We recall that, since $r\geq \rho_\Omega$, then
$g'_r(\rho_\Omega)\geq 0$, and that
$g'_r(\rho_\Omega) = 0$ if and only if $r = \rho _\Omega$.   
Assume by contradiction that $g' _r (\rho _\Omega)>0$.  
Let $x_0 \in \high (\Omega)$. Without loss of generality, assume that $x_0 = 0$. Thanks to the concavity of $g_r$, we have
\begin{equation}\label{f:conc1}
u (x) = g _r (d(x)) \leq  u (0)+  g' _ r (\rho _\Omega)  
(d(x) - \rho _\Omega)\,. 
\end{equation}
Let $p, \zeta \in \R ^n$ be associated with the point $x_0=0$ according to Theorem \ref{p:estid}, 
with $\pscal{\zeta}{p}\neq 0$,
and set
\[
\psi(x) := \pscal{p}{x}
- c \pscal{\zeta}{x}^2 +\frac{1}{2 \rho _\Omega}
|x|^2.
\]
Then, by (\ref{f:conc1}) and Theorem \ref{p:estid}, 
it holds 
\[
u(x) \leq \varphi(x) := u(0) + g_r'(\rho _\Omega)\psi(x)\,
\]
so that the function $\varphi$ touches $u$ from above. 
Some straightforward computations give 
\[
\Delta_{\infty}\varphi(0) = 
g_r'(\rho _\Omega)^3\ \Delta_{\infty}\psi(0)
= g_r'(\rho _\Omega)^3\ \left(
-2 c \pscal{\zeta}{p}^2 + \frac{1}{\rho _\Omega}|p|^2  \right)
\]
Since $g_r' (\rho _\Omega)>0$ and $\pscal{\zeta}{p} \neq 0$, it is enough to choose  $c>0$ large enough in order to have $\Delta_{\infty}\varphi(0) < -1$, contradiction.

\smallskip
\textsl{Proof of the equality $\Cut (\Omega) = \high (\Omega)$}. 

\smallskip
Since we have just proved that $u = \phi _\Omega$, we know that $u ( x) = g  ( d(x))$, with $g$ as in (\ref{defg}). 
Assume by contradiction that there exists $x_0 \in \Sigma (\Omega) \setminus \high (\Omega)$. Without loss of generality, assume that $x_0 = 0$, and set $d_0 = d (0)$. Since we are assuming $x_0 \not \in \high (\Omega)$, it holds $d_0 < \rho _\Omega$, which implies $g' (d_0) >0$. Then, we can reach a contradiction by arguing similarly as above. Namely, thanks to the concavity of $g$, we have
\begin{equation}\label{f:conc2}
u (x)  \leq u (0)+  g'  (d_0)  
(d(x) - d_0)\,. 
\end{equation}
Let $p, \zeta \in \R ^n$ be associated with the point $x_0=0$ according to Theorem \ref{p:estid}, with $\pscal{\zeta}{p}\neq 0$, and set
\[
\psi(x) := \pscal{p}{x}
- c \pscal{\zeta}{x}^2 +\frac{1}{2 d_0}
|x|^2.
\]
By (\ref{f:conc2}) and Theorem \ref{p:estid}, 
we have that
\[
u(x) \leq \varphi(x) := u(0) + g'(d_0)\psi(x)\, ,
\]
so that the function $\varphi$ touches $u$ from above.
Moreover,
\[
\Delta_{\infty}\varphi(0) = 
g'(d_0)^3\ \Delta_{\infty}\psi(0)
= g'(d_0)^3\ \left(
-2 c \pscal{\zeta}{p}^2 + \frac{1}{d_0}|p|^2  \right)\,.
\]
Since $g' (d_0)>0$ and $\pscal{\zeta}{p} \neq 0$, it is enough to choose  $c>0$ large enough in order to have $\Delta_{\infty}\varphi(0) < -1$, contradiction.  
We have thus shown that $\Sigma (\Omega ) \subseteq \high (\Omega)$. 
Since the converse inclusion holds true for all  $\Omega$, and since $M (\Omega)$ is a closed set, we conclude that the required equality $\Cut (\Omega) = \high (\Omega)$ holds.
\qed


\section{Appendix}\label{secapp}

In this Appendix we show how the proof of Theorem \ref{t:b} can be simplified under the additional assumption that the solution $u$ is differentiable. 

One can follow the proof given in Section \ref{secb} up to arriving at the equality (\ref{fgc}).   Then the conclusions
$u = \phi _\Omega$ and $\Cut (\Omega) = \high (\Omega)$ can be  achieved as follows, with no need to apply Theorem \ref{p:estid} or construct test functions.  As usual, we set for brevity $d: = d _{\partial \Omega}$.

\smallskip
\textsl{Proof of the equality $u = \phi _\Omega$}. 

\smallskip
We claim that  the function $f$ is differentiable in $(0,{\rho_\Omega} )$, and
\begin{equation}
\label{f:dsg}
f'_-(\rho _\Omega) := \lim_{t\to \rho _\Omega^-} \frac{f(t)- f(r)}{t-r} = 0.
\end{equation}

Namely, let us first show that $f$ is differentiable at an arbitrary fixed point $t_0\in (0,\rho _\Omega)$. 
Let $p_0 \in M (\Omega)$, that is, $p_0$ is the center of a ball of radius $\rho _\Omega = |p _0 - q _0|$ contained into $\Omega$, with 
$q_0 \in \pi _{\partial \Omega} ( p _0)$. We take $x_0  \in ]p_0, q_0[$ such that  
$d (x_0) = t _0$. Setting $\nu _0 := \frac{p_0 - q _0 }{\rho _\Omega}$, as $h \to 0$ there holds 
\[
u(x_0 + h \nu_0) - u(x_0) - h \pscal{\nabla u(x_0)}{\nu_0} = o(h),
\]
which implies
\[
f(t_0+h) - f(t_0) - h \pscal{\nabla u(x_0)}{\nu_0} = o(h)\,. 
\]
It remains to prove \eqref{f:dsg}.
Let $p_0$ be the center of a maximal ball as above. By Remark \ref{l:hull} and Carath\'eodory Theorem 
there exist points 
$q_0, q_1, \ldots, q_k\in ( \partial \Omega \cap \partial B_{\rho _\Omega}(p_0)   )$,
with $k$ equal at most $n$,
such that $p_0  \in {\rm conv} ( \{q_0, q_1, \ldots q_k\})$, 
i.e. 
\[
p_0 = \sum_{i=0}^k \lambda_i q_i,\qquad
\sum_{i=0}^k \lambda_i = 1, \quad
\lambda_i > 0\ \  \forall i=0,\ldots,k.
\]
Let us define the unit vectors
\[
\nu_i := \frac{p_0-q_i}{|p_0-q_i|} = \frac{p_0-q_i}{\rho _\Omega}\,, \qquad
i=0,\ldots,k.
\]
Since $u$ is differentiable at $p_0$, we have that,
for every $i=0,\ldots,k$ and $h\in (0,\rho _\Omega)$
\[
u(p_0-h\nu_i) - u(p_0) + h\, \pscal{\nabla u(p_0)}{\nu_i} = o(h),
\]
that is
\[
f(\rho _\Omega -h) - f( \rho _\Omega) + h\, \pscal{\nabla u(p_0)}{\nu_i} = o(h)\,.
\]
In turn, this equality yields
\[
f'_-(\rho_\Omega) = \pscal{\nabla u(p_0)}{\nu_i}\,,
\qquad i = 0, \ldots, k.
\]
Since $\sum_{i=0}^k \lambda_i = 1$ and $\sum_{i=0}^k \lambda_i \nu_i = 0$, we have
\[
f'_-(\rho_\Omega) = \sum_{i=0}^k \lambda_i f'_-(\rho _\Omega)  =
\sum_{i=0}^k \lambda_i\pscal{\nabla u(p_0)}{\nu_i} = 0, 
\]
which proves the claim. 

\smallskip
In view of the equality (\ref{fgc}) already proved, condition (\ref{f:dsg}) implies that $r$ is uniquely determined as $r = \rho _\Omega$, and the proof of the equality $u = \phi _\Omega$ is achieved.

\smallskip
\textsl{Proof of the equality $\Cut (\Omega) = \high (\Omega)$}. 

\smallskip
Let $x_0 \not \in \high (\Omega)$ be fixed. Then $d ( x_0 ) < \rho _\Omega$, which taking into account
the explicit expression (\ref{defg}) of the function $g$ implies
$g ' (d(x_0) )>0$. 
Using  the equality $u (x) = g (d(x))$ already proved, 
the differentiability of $u$ at $x_0$, and the inequality 
$|d(x) - d(x_0)| \leq |x-x_0|$, we see that
\[
\begin{split}
o(|x-x_0|) & =
u(x) - u(x_0)  - \pscal{\nabla u(x_0)}{x-x_0}
= g(d(x)) - g(d(x_0)) - \pscal{\nabla u(x_0)}{x-x_0}
\\ & = 
g'(d(x_0)) (d(x) - d(x_0)) + o(d(x) - d(x_0)) 
- \pscal{\nabla u(x_0)}{x-x_0}
\\ & =
g'(d(x_0)) (d(x) - d(x_0)) 
- \pscal{\nabla u(x_0)}{x-x_0} + o(|x-x_0|).
\end{split}
\]
This, combined with the inequality $g' ( d(x_0) ) >0$ noticed above,  implies that $d$ is differentiable at $x_0$.  
 We conclude that $\Sigma (\Omega) \subseteq \high (\Omega)$ and in turn, since $\high (\Omega)$ is a closed subset of $\Cut (\Omega)$, that $\Cut (\Omega) = \high (\Omega)$.
 \qed


\def\cprime{$'$}
\providecommand{\bysame}{\leavevmode\hbox to3em{\hrulefill}\thinspace}
\providecommand{\MR}{\relax\ifhmode\unskip\space\fi MR }
\providecommand{\MRhref}[2]{%
  \href{http://www.ams.org/mathscinet-getitem?mr=#1}{#2}
}
\providecommand{\href}[2]{#2}

\end{document}